\documentclass[12pt]{article}
\usepackage[utf8]{inputenc}
\usepackage{amsmath}
\usepackage{svg}
\usepackage{amsfonts,amssymb}
\usepackage{algorithm}
\usepackage{algpseudocode}
\usepackage{mathtools}
\usepackage{geometry}
\usepackage{expl3}
\usepackage[sorting=none, backend=biber]{biblatex}
\usepackage{booktabs}
\usepackage{siunitx}
\title{\vspace{-40 pt} \textbf{{A Parallelized Cutting-Plane Algorithm for Computationally Efficient Modelling to Generate Alternatives}} }
\author{Michael Lau, Filippo Pecci, Jesse D. Jenkins}

\usepackage[T1]{fontenc}
\usepackage{fancyhdr}
\usepackage[english]{babel}

\usepackage{cancel}
\usepackage{tikz}
\usetikzlibrary{shapes.geometric}
\usepackage{float}
\restylefloat{table}
\usepackage{newclude}
\usepackage{amsthm,bm,todonotes}
\newtheorem{theorem}{Theorem}

\usepackage{xcolor}

\usepackage{graphicx}
\usepackage{subcaption}
\usepackage[labelfont=bf]{caption}

\usepackage{float}
\usepackage{listings}
\usepackage{color}
\usepackage{textcomp}
\definecolor{dkgreen}{rgb}{0,0.6,0}
\definecolor{gray}{rgb}{0.5,0.5,0.5}
\definecolor{mauve}{rgb}{0.58,0,0.82}
\usepackage{minted}

\usepackage{setspace} 
\usepackage{csquotes}
\begin{document}
\maketitle
\begin{abstract}
Contemporary macro-energy systems modelling is characterized by the need to represent strategic and operational decisions with high temporal and spatial resolution and represent discrete investment and retirement decisions. This drive towards greater fidelity, however, conflicts with a simultaneous push towards greater model representation of inherent complexity in decision-making, including methods like Modelling to Generate Alternatives (MGA). MGA aims to map the feasible space of a model within a cost slack by varying investment parameters without changing the operational constraints, a process which frequently requires hundreds of solutions. For large, detailed energy system models this is impossible with traditional methods, leading researchers to reduce complexity with linearized investments and zonal or temporal aggregation. This research presents a new solution method for MGA-type problems using cutting-plane methods based on a tailored reformulation of Bender’s Decomposition. We accelerate the algorithm by sharing cuts between MGA master problems and grouping MGA objectives. We find that our new solution method consistently solves MGA problems times faster and requires less memory than existing monolithic Modelling to Generate Alternatives solution methods on linear problems, enabling rapid computation of a greater number of solutions to highly resolved models. We also show that our novel cutting-plane algorithm enables the solution of very large MGA problems with integer investment decisions.
\end{abstract}

\newpage
\section{Introduction} \label{intro}
\paragraph{} {Capacity expansion models (CEM) are optimization-based models which represent the economic capacity procurement and operations of energy systems for a given time period under different sets of assumptions, including engineering, economic, policy, and availability constraints. These models can be used in a variety of contexts to provide decision support for the planning of energy infrastructure. Examples of these insights can include visualizing the roles different energy technologies play in the energy system or how policy levers change system-wide economic incentives \cite{Ricks2022,Lombardi2020}. In recent years, the requirements these models must meet have become more demanding. Energy systems are becoming increasingly complex and coupled as electrification creates strong interconnections across sectors, weather-dependent variable renewable resources play a central role, and climate change induces extreme events. Representing these systems accurately results in very large linear or mixed-integer optimization problems, with hundreds of millions variables and constraints, which are computationally intensive \cite{vanOuwerkerk2022}. }
\paragraph{}{CEMs are formulated as least-cost optimization problems, which represent likely outcomes of efficient markets and are useful for indicative planning \cite{Ricks2022,Ricks2024,Manocha2024,Luo2024,Brinkerink2024,Bistline2023}. But least-cost optimization misses the much wider feasible region of near-optimal solutions. Prior work has shown that within a small budget relaxation, CEMs have considerable flexibility to change the composition of resources and/or the location of siting decisions, producing very different impacts along a wide range of outcomes of interest \cite{Lombardi2020,Pickering2022,Sasse2019}. Furthermore, CEMs often omit preferences or constraints important to specific stakeholders, including land-use preferences (i.e. minimizing culturally significant or biodiverse land usage), equity considerations, or tax revenue. }
\paragraph{} {Modelling to Generate Alternatives (MGA) is a method to explore the near-optimal feasible region of a convex optimization problem \cite{Brill1979, DeCarolis2011, Trutnevyte2016, DeCarolis2017,Lau2024}. It has been applied previously in energy systems CEMs to offer a wider range of feasible solutions and explore key trade-offs \cite{Lombardi2020,Pedersen2023,Pickering2022,Sasse2019}. MGA is typically carried out by first solving a least-cost optimization problem, then repeatedly re-optimizing with varying objective vectors, subject to a budget constraint limiting the relative increase in cost compared to the least-cost solution~\cite{Lau2024}. }
\paragraph{}{The key challenge is that CEMs are already very large-scale, and MGA requires the computation of hundreds of solutions to these large models for a given scenario \cite{Pedersen2023,Lau2024, Pickering2022}.  Recent work has reduced solution time for least-cost mixed-integer CEMs using Benders decomposition \cite{Jacobson2024, Pecci2024}. These methods iterate between solving a planning problem and a number of parallelized operational subproblems. Each subproblem optimizes system's operation over a short period of time (e.g., a week), while the planning problem optimizes investment decisions as well as time-coupling variables that link operation across multiple periods (e.g., weekly CO$_2$ emission budgets.) At each iteration, the solutions of the parallelized sub-problems are used to generate cuts (i.e., linear inequalities) that are added to the planning problem to improve its approximation of system's operating costs. Importantly, the validity of cuts generated by a Benders iteration only depends on the formulation of the subproblem. This property offers an opportunity to accelerate MGA in the context of CEMs, where we iteratively re-optimize models that differ only in the master problem. However, MGA problems involve a characteristic, complicating "budget constraint", which ties together planning and operational costs, requiring a tailored decomposition formulation.}
\paragraph{}{A substantial body of literature exists on utilizing decomposition methods to efficiently solve multi-objective problems. \textcite{Kagan1993} apply Benders decomposition to multi-objective electricity distribution problems, utilizing an $\epsilon-$constraint method to explore the solution space. However, work in \textcite{Kagan1993} is focused on $\epsilon-$constraints that consider either planning or operational decisions, not both at the same time. As a result, these constraints can be included either in the planning problem or operational sub-problems within a standard Benders decomposition algorithm. In MGA-type problems, we consider $\epsilon-$constraints that combine planning and operational decisions, requiring tailored decomposition schemes to deal with the resulting complicating constraints.
\textcite{Mardan2019} applied Benders decomposition to bi-objective supply chain problems but do not make any modifications to a standard Benders Decomposition in the multi-objective setting. Instead, they repeat Benders Decomposition naively while varying the weights between objectives. This approach neither includes MGA's complicating $\epsilon$-constraints nor takes advantage of repeated solution problem structure as we do. Additionally, \textcite{Raith2024} demonstrated a solution method using Benders decomposition to solve bi-objective linear programs. \textcite{Raith2024} rely on a simplex approach which does not scale well to higher numbers of objectives making it unsuitable for MGA. It also does not consider $\epsilon$-constraints, which are a defining feature of MGA. While Benders decomposition approaches have been shown to work well for standard capacity expansion problems, they have not yet been proven to work for MGA-type problems. Simultaneous to this work, \textcite{Bienstock2025} demonstrate that cutting planes can be used to warm-start repeated solutions of ACOPF problems with variations in input data. Also simultaneous to this work, \textcite{Viens2025} introduce a generalized algorithmic concept for extracting alternative solutions from Benders decomposition. They do not provide specific formulations or algorithms applied to MGA-type problems, nor do they provide computational testing.}
\paragraph{}{In this paper, we present a parallelizable cutting-plane algorithm for generating alternative, near-optimal solutions for large-scale capacity expansion problems which is faster and requires less memory than traditional monolithic solution methods. Our method decomposes system operation into small operational sub-problems that can be optimized in parallel. We accelerate convergence by sharing cuts between different MGA optimizations and partitioning the MGA space to improve the quality of cut information. We apply and demonstrate these methods on a large-scale, open-source, linear capacity expansion model and show a significant speed increase over a monolithic implementation for most problem scales. Furthermore, we demonstrate that the novel algorithm enables the solution of large-scale MGA problems with integer investment variables, a class of problems which are intractable with current methods. Applying these novel methods enables more extensive exploration of the near-optimal feasible region of capacity expansion models and similarly structured problems while retaining high model resolution and integer investment representation, thus improving the quality of decision support provided to stakeholders and decision makers.}
\paragraph{}{In Section \ref{Sec2}, we describe the problem formulation of a decomposed MGA problem and demonstrate that it is equivalent to the monolithic MGA problem. We describe methods for the sharing of resulting cuts and how solving multiple MGA iterations can allow our tailored cutting-plane algorithm to initialize with cuts from previously computed solutions. We introduce several strategies for sharing cuts and introduce a strategy for spatially partitioning MGA vectors to make retained cuts more relevant. In Section \ref{sec:numexp}, we compare the performance of the novel cutting-plane algorithm to existing parallelized monolithic solution methods on MGA capacity expansion problems with varying levels of abstraction and test the efficacy of each cut sharing method described in Section \ref{Sec2}. Finally we discuss the contributions of this paper and conclude in Section \ref{conc}.}
\section{Methodology}\label{Sec2}
\paragraph{}{In this Section, we propose a tailored cutting-plane method resulting in a parallelized, computationally efficient, MGA algorithm for energy system capacity expansion problems. First, in subsection \ref{mono}, we provide a brief review of monolithic MGA schemes. Then, in subsection \ref{sec:BendersMGA}, we discuss an existing Benders decomposition solution algorithm and introduce an equivalent reformulation of the MGA problem, whose structure is exploited to define the novel parallelized cutting-plane method. We continue by discussing suitable stopping criteria for the cutting-plane method and describing the cutting-plane MGA solution algorithm. Finally, in subsection \ref{sec:CutSelection}, we introduce the concept of sharing cuts between MGA iterates to improve runtime.}
\subsection{Standard Modelling to Generate Alternatives}\label{mono}
\paragraph{} {In this study, we consider planning models where both planning and operational decisions are optimized over a set of operational periods (e.g. weeks) $p \in P$. These models are formulated as the following Linear Program (LP):
\begin{equation}
    \label{eq:LP}
    \begin{alignedat}{3}
        &\text{minimize}&\; \; &\bm{c}^T\bm{x}+\sum_{p \in P} \bm{d}_p^T \bm{y}_p\\
        &\text{subject to}& & \bm{A}_p\bm{x} + \bm{B}_p\bm{y}_p \leq \bm{b}_p, \quad \forall p \in P\\
        &&& \bm{y}_p \in Y_p,\quad \forall p \in P\\
        &&& \bm{x} \in X
    \end{alignedat}
\end{equation}
where $\bm{x} \in \mathbb{R}^n$ represents the set of planning decisions and $\mathbf{y}_p \in \mathbb{R}^m$ represents operational decisions in operational period $p \in P$. Vectors of parameters $\mathbf{c}\in \mathbb{R}^n$ and  $\mathbf{d}_p\in \mathbb{R}^m$ represent fixed and variable costs, respectively. For each operational period, $\mathbf{A}_p\in \mathbb{R}^{r \times n}$ and $\mathbf{B}_p\in \mathbb{R}^{r\times m}$ are constraint coefficient matricies, and $\mathbf{b}_p \in \mathbb{R}^r$ is the vector of right hand side coefficients. Finally, polyhedral sets $X \subseteq \mathbb{R}^n$ and $Y_p\subseteq \mathbb{R}^m$ represent linear constraints on planning and operational decisions, respectively. In the following, we refer to a solution of Problem~\eqref{eq:LP} as a \emph{least-cost solution}, which minimizes the objective function representing total system cost.}
\paragraph{}{Modelling to Generate Alternatives (MGA) aims to compute alternative feasible solutions with total costs within a specified budget \cite{Brill1979}. In the context of CEMs the primary focus is investment planning, thus we consider an MGA formulation designed to generate alternative investment portfolios \cite{DeCarolis2011}. 
To achieve this, MGA modifies the least-cost planning problem~\eqref{eq:LP} as follows. First, total system cost, originally the objective to be minimized, is constrained to be bounded by a value $\epsilon$, indicating the maximum acceptable expenditure above least-cost. All other constraints and variables are unchanged from the original Problem~\eqref{eq:LP} \cite{DeCarolis2011, Trutnevyte2013}. Second, the objective function is set to a weighted sum of variables of interest. The variables may be any variables from the problem, but are often planning variables. The weights, $\mathbf{w}\in\mathbb{R}^n$ may be strategically selected but are often randomized \cite{Lau2024}. These weighted objective functions change the optimal solutions of the problem and encourage the exploration of the feasible set of Problem~\eqref{eq:LP} in different directions. A typical MGA formulation can be written as:
\begin{equation}
    \label{eq:MMGA}
    \begin{alignedat}{3}
        &\text{minimize}&\; \; &\bm{w}^T\bm{x}\\
        &\text{subject to}& & \bm{c}^T\bm{x}+\sum_{p \in P} \bm{d}_p^T \bm{y}_p \leq \epsilon \\
        &&& \bm{A}\bm{x} + \bm{B}_p\bm{y}_p \leq \bm{b}_p, \quad \forall p \in P\\
        &&& \bm{y}_p \in Y_p,\quad \forall p \in P\\
        &&& \bm{x} \in X
    \end{alignedat}
\end{equation}
In the MGA algorithm, we first find a least-cost solution $(\bm{x}^*,(\bm{y}_p^*)_{p \in P})$ solving Problem~\eqref{eq:LP}, to determine the desired value of the cost budget: 
\begin{equation}
\epsilon = (1+\beta)\Big(\bm{c}^T\bm{x}^* + \sum_{p \in P}{\bm{d}_p^T\bm{y}_p^*}\Big),
\end{equation}
where slack parameter $\beta>0$ is the maximum acceptable relative increase in budget compared to the least-cost.
Then, any MGA iterates with varying objective weights are solved using the formulation in~\eqref{eq:MMGA} \cite{Lau2024, Trutnevyte2013, DeCarolis2011, Pedersen2021}.}
\paragraph{}{Computational experiments based on a 12-zone continental US electricity system planning problem with one planning stage and 8736 operational hours (an LP with 14.3 million variables and 25.6 million constraints) presented in Section~\ref{sec:numexp} suggest that the MGA Problem~\eqref{eq:MMGA} is more difficult to solve than the corresponding least-cost Problem~\eqref{eq:LP} when the budget constraint is relatively tight. To illustrate this, we report the convergence performance of the commercial solver Gurobi~\cite{GurobiOptimizationLLC2024}, when implemented to solve two examples of Problems~\eqref{eq:LP} and~\eqref{eq:MMGA} using the Barrier algorithm.}
\paragraph{}{In particular, Figure~\ref{fig:Primal_feas} reports the solver's progress in reducing the feasibility residual, which measures how close the current guess is to satisfying all constraints (i.e., being feasible). We observe that when applied to the MGA problem~\eqref{eq:MMGA} with a strict budget constraint using $\beta = 0.1$, the barrier algorithm converges much more slowly compared to the least-cost problem~\eqref{eq:LP} or an MGA problem with a looser budget constraint with $\beta = 2.0$. This performance can be explained by noting that the tighter cost budget constraint in Problem~\eqref{eq:MMGA} significantly restricts the feasible space to be close to the least-cost solution of Problem~\eqref{eq:LP}. Hence, when solving Problem~\eqref{eq:MMGA}, the barrier method iterates close to the boundary of its feasible set, which shortens step length and slows convergence of the algorithm. }
\begin{figure}[H]
    \centering
    \includegraphics[width=\linewidth]{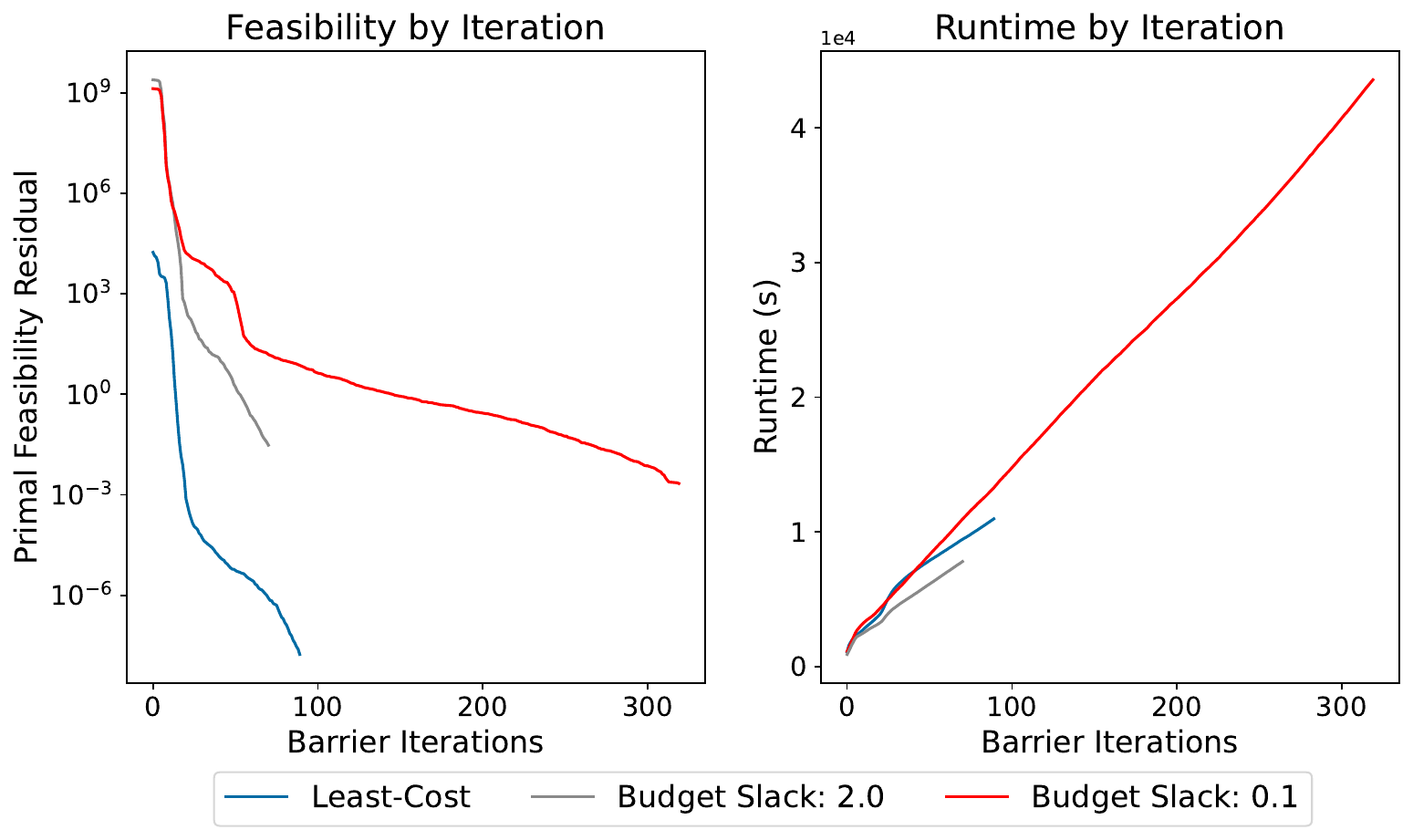}
    \caption{Convergence of primal feasibility in numerical examples for a 12-zone, continental US electricity system planning problem with one planning stage and 8,736 operational hours (an LP with 14.3 million variables, 25.6 million constraints)  with least-cost and MGA formulations with $\beta =0.1$ and $2.0$. The MGA problem with the tighter budget converges much more slowly, indicating increased computational effort relative to least-cost solutions or loose budget constraints.}
    \label{fig:Primal_feas}
\end{figure}
\paragraph{}{In order to extensively explore the near-optimal space of planning decisions, it is necessary to solve several hundred versions of Problem~\eqref{eq:MMGA} with different objective function weights \cite{Lau2024}. Hence, as shown by the experiments reported in Section~\ref{sec:numexp}, the computational capabilities of off-the-shelf optimization solvers represent a major bottleneck for the application of MGA to large planning problem with high temporal resolution and engineering detail.
As we discuss in the following, decomposition and cutting-plane methods offer a scalable approach for solving MGA problems by taking advantage of distributed computing resources through parallelization.}
\subsection{Parallelized Cutting-Plane Algorithm to Generate Alternatives} \label{sec:BendersMGA}
\paragraph{}{MGA analysis is initialized by computing a solution of the least-cost problem~\eqref{eq:LP}. Decomposition algorithms based on cutting-plane methods~\cite{Nesterov2004} (e.g., Benders Decomposition) offer scalable solution approaches to exploit the structure of Problem~\eqref{eq:LP} \cite{Jacobson2024,Pecci2024,Rahmaniani2017}. Here, we briefly summarize the key steps of a Benders decomposition algorithm applied to solve Problem~\eqref{eq:LP}. At iteration $k \geq 1$, for all $p \in P$, solve the cut-generating sub-problem:
\begin{equation}
    \label{eq:operations_sub}
    \begin{alignedat}{3}
        f_p(\bm{x}^k)=&\text{minimize}&\; \; &\bm{d}_p^T \bm{y}_p\\
        &\text{subject to}& & \bm{A}_p\bm{x} + \bm{B}_p\bm{y}_p \leq \bm{b}_p\\
        &&& \bm{y}_p \in Y_p\\
        &&& \bm{x}=\bm{x}^k \quad (\bm{\lambda}_p^k)
    \end{alignedat}
\end{equation}
In this study, we assume that suitable slack variables and penalty terms are included in Problem~\eqref{eq:operations_sub} to ensure it is always feasible. An upper bound on the optimal value of Problem~\eqref{eq:LP} is given by \begin{equation}
    \label{eq:ub}
    \text{UB}^k = \min_{j=1,\ldots,k}\sum_{p\in P}f_p(\bm{x}^j).
\end{equation}
Using the solution of the sub-problems at all previous iterations, define a piecewise linear approximation of $f_p(\cdot)$:
\begin{equation}
\label{eq:piecewise_app}
    f_p(\bm{x}) \geq \max_{j=1,\ldots,k} \Big(f_p(\bm{x}^j) + (\bm{x}-\bm{x}^j)^T\bm{\lambda}^j_p\Big)
\end{equation}
where $\bm{\lambda}^j_p$ is the vector of Lagrangian multipliers obtained solving Problem~\eqref{eq:operations_sub}. Each segment of the piecewise approximation in Equation~\eqref{eq:piecewise_app} is referred to as a \emph{cut}. A lower bound on the optimal value of Problem~\eqref{eq:LP} is obtained minimizing the piecewise linear approximation:
\begin{equation}
    \label{eq:app_least_cost_master}
    \begin{alignedat}{3}
        \text{LB}^k = &\text{minimize}&\; \; &\bm{c}^T\bm{x}+\sum_{p \in P}\theta_p\\
        &\text{subject to}& & \theta_p \geq f_p(\bm{x}^j) + (\bm{x}-\bm{x}^j)^T\bm{\lambda}^j_p \quad \forall j=1,\ldots,k \\
        &&& \bm{x} \in X
    \end{alignedat}
\end{equation}
where $\theta_p$ is an auxiliary variable that approximates the value of $f_p(\bm{x})$, for every operational period $p\in P$. Then, if the optimality gap $(\text{UB}^k-\text{LB}^k)/\text{LB}^k$ is smaller than a given convergence tolerance, the algorithm stops. Otherwise, we select a new iterate $\bm{x}^{k+1}$ and proceed to the next iteration. }
\paragraph{}{If investment variables $\bm{x}$ are constrained to be integers, two additional steps are required. The problem is first solved with linear investment variables. Once the linear relaxation is solved to convergence, the integer investment constraints are reintroduced, $\text{UB}$ and $\text{LB}$ are reset, and the problem is solved to convergence again. This process is described in \textcite{Pecci2024}. The implementation of a Benders decomposition algorithm to solve Problem~\eqref{eq:LP} is summarized in Algorithm~\ref{alg:least_cost_benders}.}
\begin{algorithm}[H]
    \caption{Benders Decomposition for Least-Cost Problem}\label{alg:least_cost_benders}
    \begin{algorithmic}
        \Require $\delta_{\text{LS}}\geq 0$, $K_{\text{LS}}\geq 1$
        \State{Define $\bm{x}^{1}$ as the solution of Problem~\eqref{eq:app_least_cost_master} with no cuts}
        \State{Define $\text{INT}$ as false, relax integrality constraints}
            \For {$k = 1,\dots,K_{\text{LS}}$}
            \For{$p \in P$, in parallel}
                \State{Solve cut-generating sub-problem~\eqref{eq:operations_sub} and compute $f_p(\bm{x}^k)$ and $\bm{\lambda}^k_p$}
            \EndFor
            \State{Define $\text{UB}^k$ as in~\eqref{eq:ub} and update cuts}
            \State{Solve Problem~\eqref{eq:app_least_cost_master} to compute $\text{LB}^k$}
            \If{$(\text{UB}^k-\text{LB}^k)/\text{LB}^k \leq \delta_{\text{LS}}$}
            \If{$\text{INT}$ == false}
            \State{Add integrality constraints, reset $\text{UB} = \text{inf}$, set $\text{INT}$ to true }
            \Else
            \State{Terminate with the planning solution corresponding to $\text{UB}^k$}
            \EndIf
            \Else 
            \State{Select new iterate $\bm{x}^{k+1}$ and go to next iteration}
            \EndIf
            \EndFor
    \end{algorithmic}
\end{algorithm}
\paragraph{}{The selection of the next iterate $\bm{x}^{k+1}$ can follow different strategies. In standard Benders decomposition schemes based on the Kelley's method \cite{KelleyJr1960}, $\bm{x}^{k+1}$ is defined as the solution of Problem~\eqref{eq:app_least_cost_master}. However, this method is known to suffer from instabilities, and alternative techniques have been proposed to regularize Benders decomposition~\cite{Ruszczyski1986,Lemarchal1995,Nesterov2004}. The investigation of regularization techniques for Benders decomposition algorithms is beyond the scope of this work, and we refer the interested reader to \cite{Zverovich2012} for general LPs, and~\cite{Gke2024,Pecci2024} for applications to energy system planning problems.}
\paragraph{}{In the remainder of this paper, we assume that the least-cost Problem~\eqref{eq:LP} has been solved using Algorithm~\ref{alg:least_cost_benders}. Let $\bm{x_{\text{LS}}}^*$ be the least-cost solution computed by the Benders decomposition algorithm. We define the cost budget as:
\begin{equation}
\epsilon = (1+\beta)\Big(\bm{c}^T\bm{x_{\text{LS}}}^* + \sum_{p \in P}f_p(\bm{x_{\text{LS}}}^*)\Big)
\end{equation}
where $\beta \in [0,1]$ is a slack parameter representing the percentage increase in system cost that we are willing to tolerate. We also assume that multiple MGA vectors of weights $\bm{w} \in \mathcal{W}$ have been identified based on one of the available techniques published in previous MGA literature \cite{Lau2024, DeCarolis2011}. We need to solve Problem~\eqref{eq:MMGA} for each $\bm{w}\in\mathcal{W}$. However, the MGA Problem~\eqref{eq:MMGA} presents a different structure from Problem~\eqref{eq:LP}, where operational periods are coupled by the budget constraint rather than the objective function. Hence, Algorithm~\ref{alg:least_cost_benders} can not be directly applied to the MGA problem. In the following we derive a cutting-plane method for solving Problem~\eqref{eq:MMGA}.}
\paragraph{}{First, we consider an equivalent reformulation of the MGA Problem~\eqref{eq:MMGA} given by:
\begin{equation}
    \label{eq:master}
    \begin{alignedat}{3}
        &\text{minimize}&\; \; &\bm{w}^T\bm{x}\\
        &\text{subject to}& & \bm{c}^T\bm{x}+\sum_{p \in P}f_p(\bm{x}) \leq \epsilon \\
        &&& \bm{x} \in X
    \end{alignedat}
\end{equation}
The following Theorem holds.}
\begin{theorem}
    Problems~\eqref{eq:MMGA} and \eqref{eq:master} are equivalent. In particular:
    \begin{enumerate}
        \item Let ($\mathbf{x}^*, \mathbf{y}^*_{p \in P}$) be an optimal solution of Problem \ref{eq:MMGA}. Then, $\mathbf{x}^*$ is optimal solution of Problem \ref{eq:master}
    
        \item Let $\bm{\hat{x}}$ be an optimal solution of Problem~\eqref{eq:master}. Then, we have that ($\bm{\hat{x}}, \bm{\hat{y}}_{p \in P}$) is optimal solution of Problem \eqref{eq:MMGA}, where:
        \begin{equation*}
            \bm{\hat{y}}_{p} \in \text{argmin}\Big(\bm{d}_p^T \bm{y}_p \; | \; \bm{A}_p\bm{x} + \bm{B}_p\bm{y}_p \leq \bm{b}_p, \: \bm{y}_p \in Y_p \Big)\quad \forall p \in P
    \end{equation*}
    \end{enumerate}
\end{theorem}
\begin{proof} We first show that the two problems define the same feasible solutions. Let ($\mathbf{x}^*, \mathbf{y}^*_{p \in P}$) be a feasible solution for Problem~\eqref{eq:MMGA}. Since it satisfies the cost budget constraint in Problem~\eqref{eq:MMGA}, the definition of function $f_p(\cdot)$ implies:
        \begin{equation}
            \bm{c}^T\bm{x}^*+\sum_{p \in P}f_p(\bm{x}^*) \leq \bm{c}^T\bm{x}^*+\sum_{p \in P}\bm{d}_p^T\bm{y}^*_p\leq \epsilon
        \end{equation}
Hence, $\bm{x}^*$ satisfies the cost budget constraint in Problem~\eqref{eq:master}. Because all other constraints in Problem~\eqref{eq:master} are the same as Problem~\eqref{eq:MMGA}, we conclude that $\bm{x}^*$ is feasible for Problem~\eqref{eq:master}. 
Now, let $\bm{\hat{x}}$ be a feasible solution for Problem\eqref{eq:master}.
By definition of function $f_p(\cdot)$, there exists $\bm{\hat{y}}_{p}\in \text{argmin}\Big(\bm{d}_p^T \bm{y}_p \; | \; \bm{A}_p\bm{\hat{x}} + \bm{B}_p\bm{y}_p \leq \bm{b}_p, \: \bm{y}_p \in Y_p \Big)$, for all $p \in P$, such that:
        \begin{equation}
             \bm{c}^T\bm{\hat{x}} +\sum_{p \in P}\bm{d}_p^T\bm{\hat{y}}_p = \bm{c}^T\bm{\hat{x}} +\sum_{p \in P}f_p(\bm{\hat{x}}) \leq \epsilon
        \end{equation}
Therefore, ($\bm{\hat{x}}, \bm{\hat{y}}_{p \in P}$) is feasible for Problem~\eqref{eq:MMGA}.

\vspace{1em} \noindent Since Problems~\eqref{eq:MMGA} and \eqref{eq:master} have the same objective functions, we conclude that an optimal solution $\bm{x}^*$ to Problem~\eqref{eq:MMGA} is optimal for Problem~\eqref{eq:master} and vice versa.
\end{proof}
\paragraph{}{We are now ready to introduce the cutting-plane method for solving Problem~\eqref{eq:MMGA}. Analogously to the least-cost Benders decomposition algorithm, at iteration $k\geq 1$, we solve the cut-generating sub-problems~\eqref{eq:operations_sub} where $\bm{x}=\bm{x}^{k}$ and check if the termination criterion is satisfied. If not, solve the approximated MGA problem, where $\theta_p$ is the piecewise linear approximation of the subproblem cost as a function of planning decisions:
\begin{equation}
    \label{eq:app_master}
    \begin{alignedat}{3}
        &\text{minimize}&\; \; &\bm{w}^T\bm{x}\\
        &\text{subject to}& & \bm{c}^T\bm{x}+\sum_{p \in P}\theta_p \leq \epsilon \\
        &&& \theta_p \geq f_p(\bm{x}^j) + (\bm{x}-\bm{x}^j)^T\bm{\lambda}^j_p \quad \forall j=1,\ldots,k \\
        &&& \bm{x} \in X,
    \end{alignedat}
\end{equation}
obtaining a new vector of planning decisions $\bm{x}^{k+1}$, and proceed to the next iteration. To determine a stopping criterion for the cutting-plane method, we note that the definition of the piecewise linear approximation~\eqref{eq:piecewise_app} implies that Problem~\eqref{eq:app_master} is a relaxation of~\eqref{eq:master}. Hence, the optimal value of the approximated Problem~\eqref{eq:app_master} provides a lower bound on the optimal value of the original Problem~\eqref{eq:master}. 
The generation of a valid upper bound for the MGA Problem~\eqref{eq:master} requires finding a planning solution that satisfies all constraints, including the nonlinear cost budget constraint. Let $w^*$ be the optimal value of Problem~\eqref{eq:master} and let $\bm{\hat{x}}\in X$ be a solution of the approximated MGA Problem~\eqref{eq:app_master}. }
\paragraph{}{If $\bm{\hat{x}}$ satisfies the MGA budget constraint, i.e. $\bm{c}^T\bm{\hat{x}}+\sum_{p \in P}f_p(\bm{\hat{x}}) \leq \epsilon$, then $\bm{\hat{x}}$ is feasible for Problem~\eqref{eq:master} and such that $\bm{w}^T\bm{\hat{x}} \leq w^*$. Hence, $\bm{\hat{x}}$ is an optimal solution of Problem~\eqref{eq:master}. As a consequence, we terminate the cutting-plane algorithm when the current solution of the Problem~\eqref{eq:app_master} approximately satisfies the nonlinear cost budget constraint.}
\paragraph{}{Similarly to the least-cost problem, we can solve problems with integer master problem decisions using a two-stage process. First, we relax the integrality constraints, resulting in a linear relaxation of the original problem. We solve that relaxation to convergence, then reset $\text{UB}$ and re-introduce the integrality constraints before re-solving to convergence with the MGA budget. The resulting algorithm is shown below in \ref{alg:cutting_plane}.}
\begin{algorithm}[H]
    \caption{Cutting-Plane iterations for solving MGA problem}\label{alg:cutting_plane}
    \begin{algorithmic}
        \Require $\bm{w}$, $\epsilon \geq 0$, $\delta_{\text{MGA}}\geq 0$, $K_{\text{MGA}}\geq 1$, $\text{INT}$ = false
            \State{Let $\bm{x}^1$ be the solution of the initial formulation of Problem~\eqref{eq:app_master}}
            \State{Relax integrality constraints}
            \For {$k = 1,\dots,K_{\text{MGA}}$}
            \For{$p \in P$, in parallel}
                \State{Solve cut-generating sub-problem~\eqref{eq:operations_sub} and compute $f_p(\bm{x}^k)$ and $\bm{\lambda}^k_p$}
            \EndFor
            \If{$\bm{c}^T\bm{x}^k + \sum_{p\in P}f_p(\bm{x}^k) \leq \epsilon(1+\delta_{\text{MGA}})$}
            \If{$\text{INT}$ == False}
            \State{Add integrality constraints, reset $\text{UB} = \text{inf}$, set $\text{INT}$ to true}
            \Else
            \State{ Terminate with solution $\bm{x}^k$}
            \EndIf
            \Else 
            \State{ Update and solve Problem~\eqref{eq:app_master} to compute new iterate $\bm{x}^{k+1}$}
            \EndIf
            
            \EndFor
    \end{algorithmic}
\end{algorithm}
\paragraph{}{The slack parameter $\beta$ in the budget constraint is somewhat arbitrary. In practice, convergence tolerance $\delta_{\text{MGA}}$ does not need to be excessively tight as the strict satisfaction of the cost budget constraint does not represent any structural or physical property of the system. In this paper, we set $\delta_{\text{MGA}}=0.5\%$ as we find that this value represents a good compromise between accuracy and computing time for the considered case study. }
\paragraph{}{Finally, we observe that the cut-generating sub-problems~\eqref{eq:operations_sub} are independent and can be solved in parallel in both Algorithms~\ref{alg:least_cost_benders} and~\ref{alg:cutting_plane}. As shown in Section~\ref{sec:numexp}, this allows us to exploit distributed computing resources to achieve unprecedented computational performance for MGA studies. We also note that these parallelized operational subproblems can be used to rapidly evaluate interior solutions generated via interpolation in post-processing, as in \textcite{Lau2024b}.}
\paragraph{}{Note that the cuts computed during the application of Algorithm~\ref{alg:least_cost_benders} to solve the least-cost Problem~\eqref{eq:LP} provide a lower bound on subproblem cost $f_p(\bm{y})$  as a function of planning decisions. Since the budget constraint in the MGA Problem~\eqref{eq:master} involves the same functions $f_p(\bm{y})$, those cuts can also be used to form the piecewise linear relaxation of the budget constraint in the approximated MGA Problem~\eqref{eq:app_master}. Hence, we can take advantage of the previous Benders decomposition iterations and initialize the approximate MGA Problem~\eqref{eq:app_master} with the cuts generated during Algorithm~\ref{alg:least_cost_benders} or prior iterations of Algorithm \ref{alg:cutting_plane}. As shown in Section~\ref{sec:numexp}, this significantly accelerates the convergence of the cutting-plane method. The overall parallelized MGA cutting-plane algorithm is presented in Algorithm~\ref{alg:MGA}.}
\begin{algorithm}[H]
    \caption{Parallelized Cutting-Plane To Generate Alternatives (CGA)}\label{alg:MGA}
    \begin{algorithmic}
        \Require $\beta \in [0,1]$, $\delta_{\text{LS}}\geq 0$, $K_{\text{LS}}\geq 1$, $\delta_{\text{MGA}}\geq 0$, $K_{\text{MGA}}\geq 1$
        \State{Implement Algorithm~\ref{alg:least_cost_benders} to solve the least-cost problem~\eqref{eq:LP}}
        \State{Let $\bm{x}_{\text{LS}}^*$ be the solution of the least-cost problem and define $$\epsilon = (1+\beta)\Big(\bm{c}^T\bm{x_{\text{LS}}}^* + \sum_{p \in P}f_p(\bm{x_{\text{LS}}}^*)\Big)$$}
        \State{Generate set of MGA weight vectors $\bm{w}\in \mathcal{W}$}
        \For {$\bm{w} \in \mathcal{W}$, parallelized}
            \State{Initialize Problem~\eqref{eq:app_master} with the cuts computed by Algorithm~\ref{alg:least_cost_benders}}
            \State{Implement Algorithm~\ref{alg:cutting_plane} to solve the MGA problem~\eqref{eq:master}}
            \State{Possibly reset cuts in Problem~\eqref{eq:app_master} based on the strategy from Section~\ref{sec:CutSelection}}
        \EndFor
    \end{algorithmic}
\end{algorithm}
\subsection{Cut sharing}\label{sec:CutSelection}
\paragraph{}{The standard Benders decomposition approach tends to suffer from instabilities which result in long solution times. Instabilities arise from a lack of sufficient information from subproblem cuts in early iterations of the master problem resulting in a very poor lower bound approximation of subproblem costs as a function of planning decisions. It follows that the resulting subproblem cuts are also uninformative, as their solutions assume unreasonable planning decisions far from the optimal solution. In the least-cost problem, we utilize a level-set regularization to avoid such instabilities \cite{Pecci2024}. Regularization moves master problem solutions into the interior of the space, which improves planning decisions when the subproblem cost as a function of planning decisions is not well known and thus limits instabilities, accelerating convergence. }
\paragraph{}{We observe a similar problem in the parallelized MGA cutting-plane method introduced in Section \ref{sec:BendersMGA}. Performing computational experiments on a 26-zone, 8736-hour example system with 30.1 million variables and 53.7 million constraints, we report the total system-cost convergence of the Cutting-Plane to Generate Alternatives (CGA) algorithm for an example MGA problem with no cuts retained from prior solutions in Figure \ref{fig:NoCuts}. We observe that the problem converges slowly, with total system costs an order of magnitude above the budget for many iterations. Additionally, we note that iteration time increases with each iteration. This well-documented behavior of cutting-plane methods is due to the row-generating nature of the algorithm which increases the size of the master problem at each iteration. In particular, our multi-cut Algorithm~\ref{alg:cutting_plane} adds $P$ cuts per iteration, which can exacerbate this phenomenon. }
\begin{figure}[H]
    \centering
    \includegraphics[width=\linewidth]{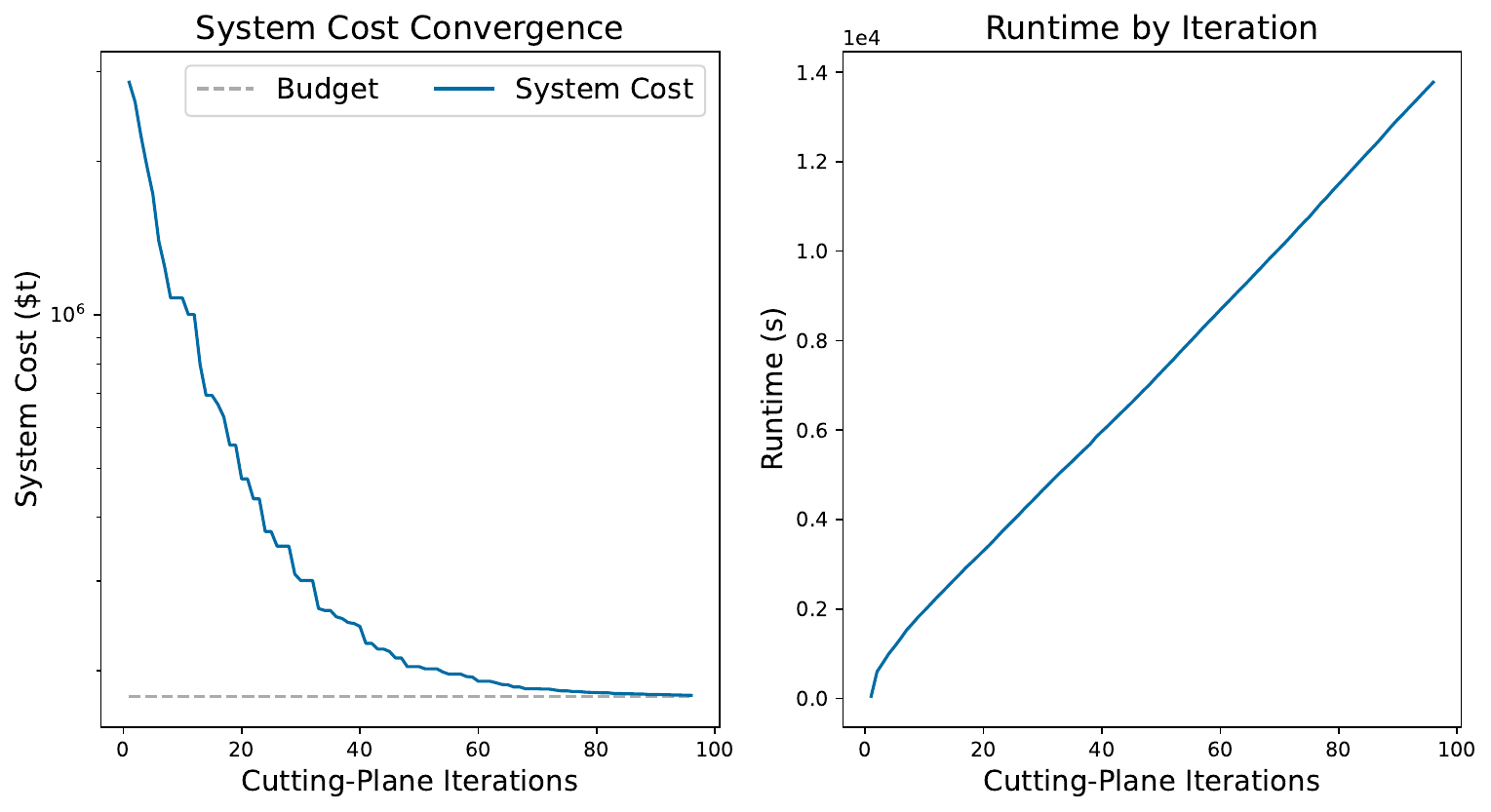}
    \caption{Convergence of Total System Cost in example solutions of 26-zone, 8736-hour MGA problem with 30.1 million variables and 53.7 million constraints solved with CGA and no cuts retained between iterations. The problem converges slowly, indicating a series of unrealistic master problem solutions. Solution time per iteration also increases, as cutting-plane methods increase master problem size at each iteration.}
    \label{fig:NoCuts}
\end{figure}

\paragraph{}{To solve the instability problem in the MGA context, we introduce a set of cut-sharing schemes. We note in Section \ref{sec:BendersMGA} that cuts generated during the least-cost solution define lower bounds for $f_p(\cdot)$ which remain valid in the context of the MGA master problem (\ref{eq:app_master}) because the subproblems (\ref{eq:operations_sub}) and the variables of the master problem do not change between the least-cost solution and the first MGA iteration. The same logic holds between MGA iterations. Thus, each MGA master problem can be initialized with previously generated cuts, providing additional precomputed information about operational costs as a function of master problem decisions
 to the problem. }
\subsubsection{Cut Sharing}
\paragraph{}{Cut sharing strategies must balance the quality of information retained with increasing master problem size. We describe three potential strategies here, with testing in Section \ref{sec:numexp}. These are:
\begin{enumerate}
    \item Retain cuts from the least-cost solution only between MGA iterations, forget all other cuts generated.
    \item Retain all cuts from all MGA iterations and the least-cost solution.
    \item Retain the first $n$ cuts generated between the least-cost and early MGA iterations.
\end{enumerate}
We selected these strategies to show the tradeoff in sharing cuts - too many and the master problem will become too slow, too few and the decomposition may need to run many more iterations. Please note that many other strategies are possible. An important area prime for future work is determining optimal cut sharing/pruning strategies.}
\subsubsection{Objective Partitioning}
\paragraph{}{To maximize the usefulness of sharing cuts between MGA iterations, we explore a strategy to partition the exploration of the capacity space represented in the master problem. As discussed in Section \ref{sec:BendersMGA}, Benders cuts form piecewise-linear approximations of the subproblem cost. Recall the standard Benders cut formulation presented in \ref{eq:app_master}. Each Benders cut is equivalent to the subproblem cost when the optimal capacity decisions from the master problem (\ref{eq:app_master}) or (\ref{eq:app_least_cost_master}) are equivalent to the set of capacity decisions which generated the cut. When the deviation between the new optimal solution of the master problem, $\bm{x}^*$ and the set of decisions which generated the cut in question increases, we expect the approximation of subproblem costs in the master problem  to become less accurate.}
\paragraph{}{As a result of this relationship, the piecewise approximation \eqref{eq:piecewise_app} of subproblem costs as a function of capacity is most accurate in heavily sampled regions of the capacity space. In least-cost settings, it is difficult to take advantage of this property of piecewise-linear approximations of convex functions. In the MGA context, however, we have a means to direct the search of the capacity space: the objective coefficient vector. Thus, we can partition the set of MGA vectors by directionally clustering them using a clustering algorithm, then run each of the resulting clusters in separate algorithm instances. As we will demonstrate in section \ref{sec:numexp}, objective partitioning with cut sharing between MGA iterations offers computational advantages.} 
\paragraph{}{Observe that this does requires some additional computation before running each individual CGA instance. Naively creating a set of MGA vectors ahead of time then clustering them is only possible with some vector generation based MGA methods such as Random Vector, Variable Min/Max, or some combination thereof \cite{Lau2024, Berntsen2017,Trutnevyte2013}. Implementing a similar objective partitioning method will also work for geometrically derived methods, such as Modelling All Alternatives but will require some additional algorithmic creativity \cite{Pedersen2021,Grochowicz2023}.}

\section{Numerical Experiments} \label{sec:numexp}
\paragraph{}{In this section, we present several numerical experiments to evaluate the application of the Cutting-Plane to Generate Alternatives (CGA) Algorithm~\ref{alg:MGA} on electricity system capacity expansion models with different sizes and levels of computational complexity.}
\paragraph{}{First, we introduce electricity system capacity expansion models, the class of model we test the algorithm on in subsection \ref{subsec:cem}. Then in subsection \ref{compsetup}, we describe the computational setup used to conduct the testing. In subsection \ref{tests}, we describe the testing regime carried out to characterize the scaling of cutting-plane and monolithic MGA runtime with cores and zones, and the results of those. We also describe and share results from tests characterizing how cutting-plane runtime scales with the number of subproblems, and tests comparing cut-sharing strategies. Finally, we share brief results comparing monolithic and cutting-plane algorithm performance on problems with mixed-integer investment variables.}
\subsection{Generating Alternatives for Electricity System Capacity Expansion Models}\label{subsec:cem}
\paragraph{}{We specifically demonstrate the application of the CGA Algorithm~\ref{alg:MGA} to electricity system capacity expansion models. Capacity Expansion Models (CEMs) are optimization problems which optimize the planning and operational decisions of an electricity system over a given time period, typically a year, modeled with hourly resolution. CEM decision variables include investment decisions for different generation technologies, storage, and transmission in each zone, charging and discharging of storage, unit commitment decisions for thermal power plants, dispatch quantity decisions for all generators, and energy flows between regions in each operational timestep. CEMs typically include constraints representing the technical characteristics of the system including power balance requirements, transmission constraints, ramping constraints, storage charge and discharge rates, hydropower streamflow constraints, and policy constraints when desired. CEMs are generally solved with a system-cost minimization objective, which is equivalent to maximizing social welfare for the whole system assuming perfect competition, complete markets without market failures, and rational actors. }
\paragraph{}{The MGA formulation and methodology presented in Section~\ref{Sec2} apply to problems that have the same structure as Problem~\eqref{eq:LP}, e.g., planning decisions that must be co-optimized with numerous operational decisions over independent operational periods. Previous works~\textcite{Jacobson2024} and~\textcite{Pecci2024} have shown that electricity system capacity expansion models belong to this class. In fact, they can be written as Problem~\eqref{eq:LP} by representing time-linking constraints between operational periods (e.g., weeks) through coupling variables that are included in the vector of planning decisions $\bm{y}$. Considered linking constraints include policy constraints (e.g., CO$_{2}$ emission caps) and multi-day energy storage resources (e.g., hydropower reservoirs). By leveraging these reformulations, we apply Algorithm~\ref{alg:MGA} to perform MGA on capacity expansion models with high temporal resolution and engineering detail, limiting the use of abstractions to reduce problem size. The resulting parallelized MGA algorithm is capable of achieving unprecedented computational performance for models with this size and level of detail.}
\subsubsection{Test cases} \label{models}
\paragraph{}{To test the efficacy of the cutting-plane algorithm presented in this paper, we carried out a set of numerical experiments using the GenX open-source electricity system planning model \cite{Jenkins2017}. We compute MGA solutions to 3, 7, 12, 16, 22, and 26 zone models of the contiguous United States as generated by the open-source data compilation software PowerGenome \cite{Schivley2021}. All models consider 52 operational weeks with hourly resolution (8736 total time steps) and linearized unit commitment. Table \ref{tab:Tab1} shows the size of each test problem.}
\begin{table}[H]
    \centering
    \caption{Example System Descriptions} \label{tab:Tab1}
\begin{tabular}{SSSS}

     {Zones} & {Generator Clusters} & {Variables ($\times 10^6$)} & {Constraints ($\times 10^6)$}\\ \midrule 
    3 & 214 & 3.7 & 6.5  \\
    7 & 485 & 8.6 & 15.3 \\
    12 & 796 & 14.3 & 25.6 \\
    16 & 970 & 18.7 & 33.4 \\
    22 & 1381 & 25.7 & 45.8  \\
    26 & 1613 & 30.1 & 53.7\\
\end{tabular}
\end{table}
\paragraph{}{In all tests except for iteration scaling, each least-cost problem is solved once and each MGA problem is solved for 16 MGA iterations (i.e., generating 16 alternative, near optimal, feasible solutions). We arbitrarily chose to use 16 MGA iterations as this resulted in a sufficient number of problems to understand performance characteristics, without exceeding the computational time limits of our computer cluster. }
\paragraph{}{In the comparison between monolithic MGA and the cutting-plane algorithm, all MGA vectors were generated using the combination method described in \textcite{Lau2024}, with zonally aggregated Variable Min/Max runs making up 75\% of all runs and Random Vector runs making up the rest. Please note that individual MGA iterates may be more or less difficult to solve than one another due to semi-randomly generated cost vectors.}
\subsection{Computational Setup}\label{compsetup}
\subsubsection{Language, Processors, and Solver}
\paragraph{}{All algorithms and optimization problems are implemented via Julia 1.9.1 \cite{Bezanson2017}. Optimization solvers are called through JuMP 1.18.1 \cite{Lubin2023}. All problems are solved on the Della computer cluster at Princeton University on 2.8GHz Cascade Lake Processors. Due to queue constraints on the cluster, we consider a problem intractable if it takes longer than 72 hours wall-clock time or takes more than 1.2 TB of memory.}
\paragraph{}{All problems are solved using Gurobi (v10.0.1) using the barrier method \cite{GurobiOptimizationLLC2024}. Optimality tolerance is set to $1\times10^{-3}$ . For the monolithic and cutting-plane algorithm master problems, crossover is turned off, as they do not require basic solutions~\cite{Jacobson2024}. If the feasible solution returned contains negative capacity values, we re-optimize with crossover turned on. Subproblems are solved with crossover on since cut generation necessitates basic solutions.} 
\subsubsection{Monolithic Parallel Setup}
\paragraph{}{Monolithic parallel problems are set up to solve MGA iterates in parallel. Several setups were tested initially on the 3-zone problem, each with 64 cores distributed evenly between $m$ processes. In preliminary testing, we found that a setup with 4 parallel MGA iterates solved by a Gurobi instance using 16 cores offered the best combination of speed and memory usage. The least-cost solution was solved by a Gurobi instance using 16 cores.}
\subsubsection{Cutting-Plane Setup}
\paragraph{}{The cutting-plane algorithm is set up to take full advantage of the decoupled subproblems. As such, when the time domain is divided into $|P|$ subproblems, we allocated $|P|$ cores. As the time domain of these problems spans 8736 hours (52 weeks), we allocate 52 cores, with one core allocated per sub problem. At each algorithm iteration, we wait for all subproblems to finish then return cut information. The master problem solver uses up to 32 cores (automatically determined by the solver), allocated from the subproblems, while the subproblems wait for the next master problem to solve. All problems were solved on single computing nodes. In the case of the cutting-plane algorithm, all MGA iterates are run in series, though we note that it is possible and advantageous to spawn multiple instances of this algorithm in a similar way as the parallelized monolithic implementation, assuming a sufficient number of CPUs are available. The initial least-cost solution is generated using the regularized Benders decomposition algorithm presented in \textcite{Pecci2024}.}
\subsubsection{Cut Sharing and Objective Partitioning Setup}
\paragraph{}{To perform the objective partitioning tests presented in section \ref{compcutshare} we used the Variable Min/Max MGA method \cite{Trutnevyte2013}, since using the combination method used in the main Monolithic/CGA comparison tests here may result in unreasonable comparisons between clusters (i.e. if some clusters of vectors are all integer weights while others are real numbers). To generate the list of clustered vectors, we first generated 320 MGA vectors. We then clustered them using k-means clustering as implemented in Clustering.jl \cite{AlexeyStukalov2024}. This generated 6 clusters over the course of 100 iterations of the clustering algorithm. We then selected one of the resulting clusters and used the first 16 vectors from the cluster as the set of MGA vectors for each CGA instance.}
\subsection{Testing and Results} \label{tests}
\paragraph{}{To compare the algorithms tested here, we define four  metrics. \textit{MGA solution time} is defined as the average wallclock time it takes to solve a given MGA problem for a given method and setup. \textit{Runtime} is defined here as the total time required to initialize, solve, and write results from a given model, including least-cost and 16 MGA iterates. \textit{Memory} is the total memory allocation required to initialize, solve, and write results for the set of MGA iterates. }
\subsubsection{Zonal Scaling Tests}\label{ZoneScaling}
\paragraph{}{\textcite{Jacobson2024} has shown that CEM error decreases as objective resolution increases. Therefore zonal scaling of the CGA algorithm is of core concern. For this test, we run all models described in Table \ref{tab:Tab1} with both cutting-plane and the parallelized Monolithic MGA formulation. Each problem is run for 5 sets of 16 MGA iterations, resulting in a sample size of 80 MGA iterations run for each problem formulation. Please note that for these tests, all CGA runs used the least-cost cut sharing method (e.g. cuts from the least-cost solution are passed to each MGA iterate). See Table \ref{tab:zone}.}
\begin{table}[H]
    \centering
    \caption{Testing Regime: Zone Scaling}
    \label{tab:zone}
    \begin{tabular}{SSSSS}
        {Method} & {Cores} & {Parallel Subproblems} & {Parallel MGA iterates} & {Zones} \\
        \midrule
        {Cutting-Plane} & 32 & 32 & 1 & 3 \\
        {Cutting-Plane} & 56 & 56 & 1 & 7 \\
        {Cutting-Plane} & 56 & 56 & 1 & 12 \\
        {Cutting-Plane} & 56 & 56 & 1 & 16 \\
        {Cutting-Plane} & 56 & 56 & 1 & 22 \\
        {Cutting-Plane} & 56 & 56 & 1 & 26 \\
        {Mono. Par.} & 64 &n/a & 4 & 3 \\
        {Mono. Par.} & 64 &n/a & 4 & 7 \\ 
        {Mono. Par.} & 64 &n/a & 4 & 12 \\ 
        {Mono. Par.} & 64 &n/a & 4 & 16 \\
        {Mono. Par.} & 64 &n/a & 4 & 22 \\ 
        {Mono. Par.} & 64 &n/a & 4 & 26 \\ 
    \end{tabular}
\end{table}
\paragraph{}{Results from zonal scaling tests are presented in Figure \ref{fig:ZoneScaling} and average results are presented in Table \ref{tab:ResultsZone}. In Table \ref{tab:ResultsZone}, the algorithm with the best performance in each metric has their value bolded. }
\begin{figure}[H]
    \centering
    \includegraphics[width=\linewidth]{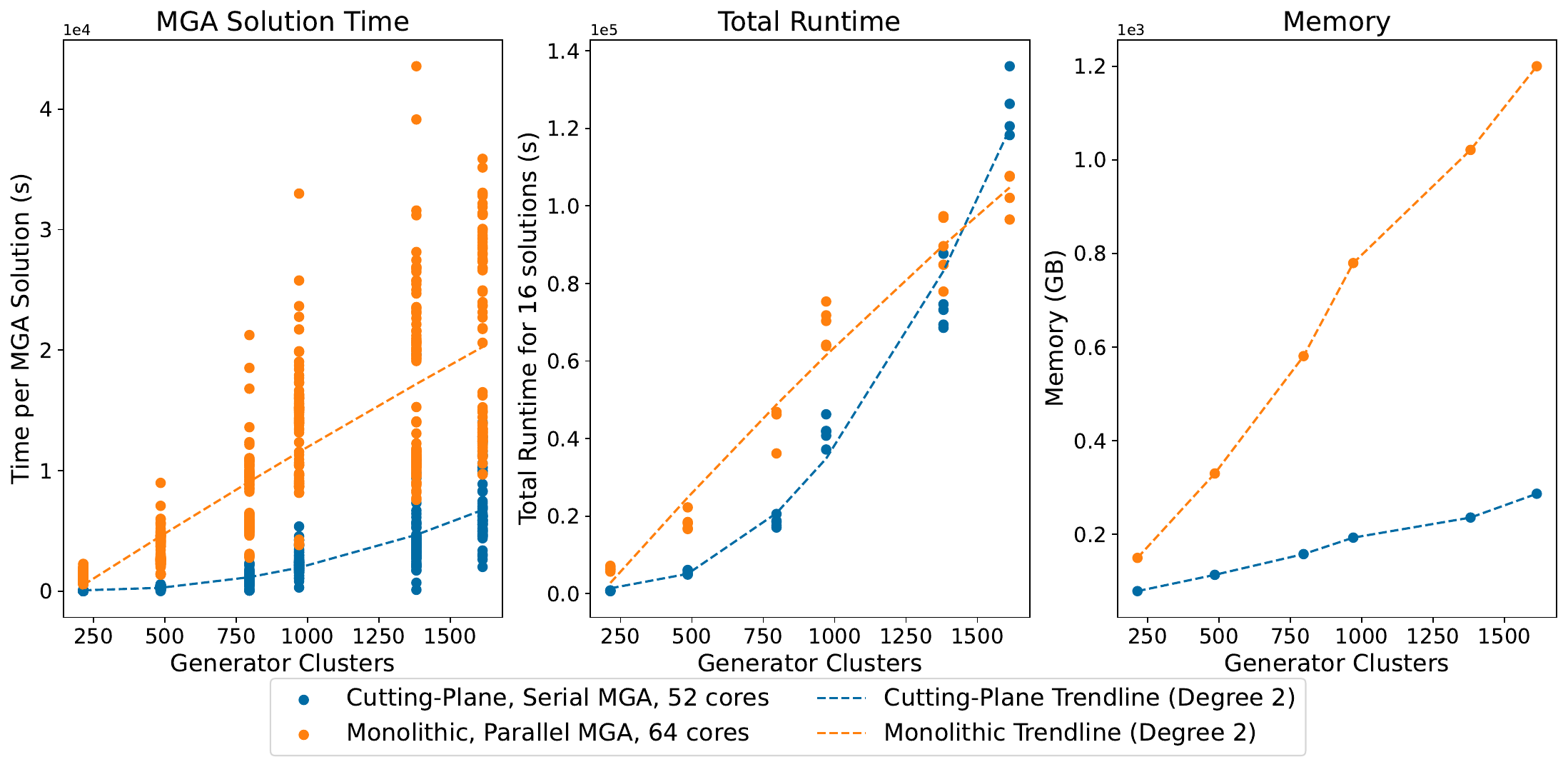}
    \caption{Scaling of 64-core parallelized monolithic MGA and 52-core cutting-plane algorithm average time per MGA solution and total runtime with zones. Each problem is run for a least-cost solution and 16 total MGA solutions. Runs tabulated in Table \ref{tab:zone}. }
    \label{fig:ZoneScaling}
\end{figure}
\begin{table}[H]
    \centering
    \caption{Results: Zonal Scaling}
    \label{tab:ResultsZone}
\begin{tabular}{ll|lll|lll}
     \multicolumn{2}{c|}{\textbf{Case}} & \multicolumn{3}{c|}{\textbf{Monolithic}} &\multicolumn{3}{c}{\textbf{CGA}} \\
    Zones & Gen. & Sol. Time & Runtime & Mem. & Sol. Time & Runtime & Mem. \\
    (\#) & (s) & (s) & (GB) & (s) & (s) & (GB) \\
    \midrule
    3 & 214 & $1190$ & $6469$ & $149$ & \textbf{34} & \textbf{748} & \textbf{78}\\
    7 & 485 & 3531 & 18476 & 330 & \textbf{303} & \textbf{5472} & \textbf{113} \\
    12 & 796 & 7765 & 43094 & 581 & \textbf{1049} & \textbf{18490} & \textbf{158}\\
    16 & 970 & 13566 & 69097 & 779 & \textbf{2364} & \textbf{41593} & \textbf{193} \\
    22 & 1381 & 16940 & 89329 & 1021 & \textbf{4276} & \textbf{74690} & \textbf{236}\\
    26 & 1613 & 20112 & \textbf{103458} & 1200 & \textbf{7014} & 125287 & \textbf{287}\\
\end{tabular}

\end{table}
\textit{Notes:} All values are means. The better value for each metric is in bold.
\paragraph{}{Our testing reveals that the CGA algorithm solves individual linear MGA iterates significantly faster than a monolithic formulation. In comparison to previous results from Benders Decomposition for least-cost problems, we note that this is a substantial change: \textcite{Jacobson2024} demonstrated that a similar Benders Decomposition scheme is faster for MILPs, but is outperformed by commercial solvers for LPs. However, the advantage conferred by CGA decreases as spatial resolution increases. Specifically, while CGA solves individual 3-zone cases 34 times faster than the monolithic implementation on average, it only solves the 26-zone case 2.8 times faster on average. This is broadly consistent with results reported in \textcite{Jacobson2024}, which showed that the monolithic implementation can scale better when solving LPs than temporal Benders Decomposition with respect to increasing spatial resolution.}
\paragraph{}{While the CGA algorithm is much faster than the monolithic formulation at solving problems in series, monolithic problems can perform well on total runtime when parallelized. In particular, the monolithic problem tends to scale better than CGA with increasing spatial resolution. When solving for 16 total MGA solutions on our smallest cases, CGA is approximately 8.5 times faster than the parallelized monolithic implementation tested here. Yet, in the most spatially resolved case, the 26-zone case, the parallelized monolithic implementation outperforms serialized CGA. This is largely because CGA as currently implemented scales poorly with spatial dimension, as the decomposition method used herein currently applies only to the temporal dimension of the problem. However, as demonstrated by \textcite{Parolin2025}, it is possible and beneficial to apply the same decomposition algorithm to a spatially decomposed formulation, which can improve spatial scaling characteristics. The CGA algorithm can be applied essentially as-is to the spatial decomposition approach in \textcite{Parolin2025}, although it may require some additional heuristics to handle excessive numbers of cuts. We recommend this as an area of future work. It is important to note that instances of CGA can also be parallelized, provided enough cores are available, though such an arrangement is not tested explicitly here.}
\paragraph{}{Memory usage is another core consideration, as memory limitations often constrain the scale of macro-energy systems planning models before runtime limitations become binding. MGA implementations in particular often become memory-constrained as they solve problems in parallel, increasing their computational demands at any given time. We find that each sequential cutting-plane algorithm implementation (including master problem, subproblems, and overhead) uses roughly the same amount of memory as an individual monolithic implementation. However, due to its increased speed, the cutting-plane algorithm has much better practical memory performance than parallelized monolithic MGA. To achieve the same throughput of problems as one serial implementation of CGA, for instance, one would need to solve four instances of the monolithic MGA problem in parallel, thus requiring four times more memory. }
\subsubsection{Cut Sharing Methods Tests}
\label{compcutshare}
\paragraph{}{To compare cut-sharing and objective partitioning methods, we tested all methods described in section \ref{sec:CutSelection} on the 26-zone, 8736-hour model, with access to 52 cores. Each of the cut sharing methods was run for 80 MGA iterations (5 sets of 16 iterations), while the no-cut sharing counterfactual was run for 13 MGA iterates due to long runtimes. Method three, which retains the first $n$ cuts, was run with $n=11000$, which was arbitrary but performed the best for our problem in simple preliminary tests comparing performance for algorithms with $n \in [7000, 9000, 11000, 13000, 15000]$. This corresponds with the cuts from the first 200 or so Benders iterations. For the 26-zone model, the least-cost problem converges after about 48 iterations, and the following MGA iterations take between 10 and 100 iterations depending on the objective vectors selected, with the mean around 30. Thus, the 11000-cut version is sharing cuts from roughly the first 5 MGA iterations to initialize the remaining problems. Developing improved methods to select which cuts to share and which to delete is a fruitful area for future study. We ran method three twice, once with objective partitioning and once without, to demonstrate the value of objective partitioning. The tests described here are summarized in Table \ref{tab:cutandspace}}
\paragraph{}{The objective partioning is carried out by pre-generating 320 MGA vectors, clustering them into 6 sets using k-means clustering as implemented in the Clustering.jl package \cite{AlexeyStukalov2024}, run for 100 iterations. We then used the first 16 vectors from the first cluster as the set of MGA vectors for each test run. In real-world applications, we suggest pre-generating and clustering a large set of MGA vectors, then divying them among parallel CGA implementations.}
\begin{table}[H]
    \centering
    \caption{Testing Regime: Cut Sharing and Objective Partitioning}
    \label{tab:cutandspace}
    \begin{tabular}{SSSS}
        {Cut Sharing Method} & {Label} & {Objective Partitioning} & {MGA iterates} \\
        \midrule
        {None} & {None} & {No} & 13 \\
        {All Cuts} & {All} & {No} & 80 \\
        {Least-Cost Cuts} & {LC} & {No} & 80 \\
        {First 11000 Cuts} & {11000, NOP} & {No} & 80 \\
        {First 11000 Cuts} & {11000, OP} & {Yes} & 80\\
    \end{tabular}
\end{table}
\paragraph{}{In the previous subsection, we showed that the base CGA method sharing cuts only from the initial least-cost problem solution can outperform a parallelized monolithic formulation with industry-standard solver on linear MGA iterates. Here we share results of testing described in Section \ref{sec:CutSelection} aimed to demonstrate the benefits of cut-sharing between MGA iterates and of objective partitioning of the MGA objective vectors to accelerate CGA. Graphical results are presented in Figure \ref{fig:CutConvBox}. Median numerical results are presented in Table \ref{tab:ResultsCuts}. Please note that the method with the best performance in each metric has that value in bold text.}
\begin{figure}[H]
    \centering
    \includegraphics[width=\linewidth]{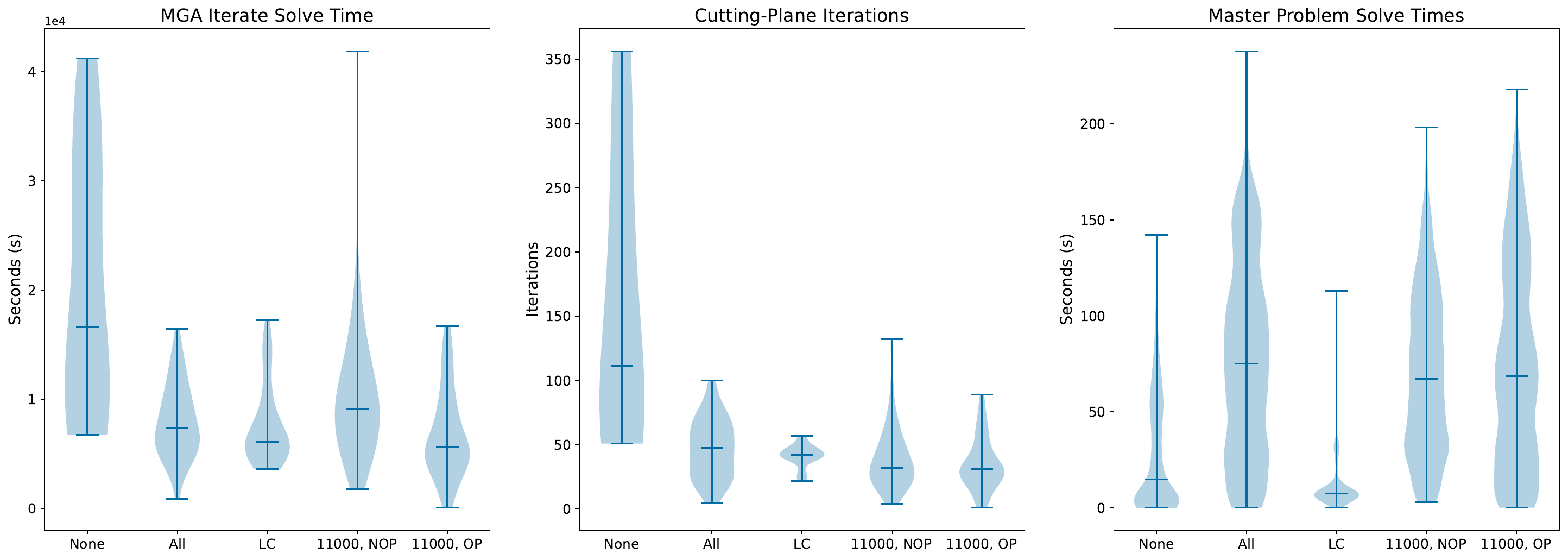}
    \caption{Comparison of cut sharing methods. From the left, plot 1 shows the distribution of runtimes to solve one MGA problem for each method, plot 2 shows the distribution of cutting-plane iterations required to solve one MGA problem for each method, and plot 3 shows the distribution of master problem solve times for each method. Note that horizontal lines show maximum, minimum, and median values for each violin. The objective partitioned, first-\textit{n} cuts method with 11000 cuts performs best in solution time and iterations required.}
    \label{fig:CutConvBox}
\end{figure}
\begin{table}[H]
    \centering
    \caption{Results: Cut Sharing and Objective Partitioning}
    \label{tab:ResultsCuts}
    \begin{tabular}{c|ccccc}
       {Metric} & None & All & Least-Cost & 11000, NOP& 11000, OP\\
       \midrule
       Sol. Time (s) & 16619 & 7380 & 6136 & 9086& \textbf{5633}\\
       Iterations & 112 & 48 & 42 & 32& \textbf{31}  \\
       Master Sol. Time (s) & 15 & 75 & \textbf{7} & 70& 69 \\
    \end{tabular}
\end{table}
{\textit{Notes:} All values are medians. The value of the algorithm with the better value for each metric is bolded.}
\paragraph{}{Cut sharing is shown to have a beneficial effect on the solve time of MGA iterates, reducing median solve time relative to the no-cut control in all test cases. Generally, we also show that sharing cuts reduces the frequency of long solve times and increases the frequency of very short solve times. However, it is not clear that naively sharing large numbers of cuts between MGA iterates is beneficial, as sharing all cuts (labeled "All'' in figures/tables) does not outperform the strategy sharing only the cuts from the initial least-cost solution (labeled "Least-Cost") in either solve time or iterations to convergence. Sharing cuts also does not exclude the existence of long solve time outlier problems, although it does decrease their likelihood, as shown in Figure \ref{fig:CutConvBox}. We note that, as hypothesized, strategies that increase the number of cuts retained result in an increase in master problem solve time, as the additional cuts expand the size of the master problem. Thus, cut-sharing strategies must balance a reduction in overall iterations to converge against an increase in master problem solve time per iteration. In general, we find the most performant cut-sharing strategy tested herein was a strategy sharing only the first 11,000 cuts with objective partition by clustering similar MGA vectors to be solved sequentially with the same cuts (labeled "11000, OP"). This method shares both the least-cost cuts and a set of initial cuts from the first MGA iteration, which provides improved information to later MGA iterates limiting the increase in master problem size. Without objective partitioning, however, the Least-Cost cut sharing methodology performed the best in terms of overall runtime, as it has very short master problem solve times which offset the additional iterations required to converge relative to the 11000 cut version (labelled "11000, NOP"). }
\paragraph{}{Partitioning MGA search vectors is shown to be beneficial for the CGA algorithm when using the first-\textit{n} cuts strategy. Importantly, partitioning is shown to dramatically decrease the minimum time to solve an MGA iterate beyond solely initializing each master problem with previous cuts. The partitioned case shows significantly more MGA iterates solving in under 20 iterations and in under 1000 seconds than the non-partitioned case for this problem, supporting the hypothesis that cuts generated by MGA iterates with similar objective vectors are more effective at accelerating convergence of the CGA algorithm. Similarly, the max number of iterations required for the algorithm to converge is lower in the partitioned case than in the control case.}
\paragraph{}{These results are achieved with a relatively simplistic MGA vector clustering method (k-means clustering), and a relatively small number of initially generated MGA vectors (320), which limits the potential efficacy of the partitioning approach since the initial 320 vectors may not be very clustered in the high-dimensional MGA space. Furthermore, we do not intelligently drop cuts from the master problem during the cutting-plane iterations, resulting in excess cut sharing and larger master problems. Exploring other clustering/partitioning methodologies and improved methods for selecting cuts to retain will be valuable avenues of future work in this area.}
\paragraph{}{To visualize the differences in performance characteristics between the cut sharing strategies tested herein, we selected a representative run of each method to plot here. Please note that for the method sharing all cuts, this was the final iteration of the 16 MGA iterates, where the master problem was fully initialized with computed cuts from the prior 15 MGA iterations plus the initial least-cost solution. This case provides the best starting information but results in significantly longer master problem solve times. The results are presented in Figure \ref{fig:CutConv}.}
\begin{figure}[H]
    \centering
    \includegraphics[width=\linewidth]{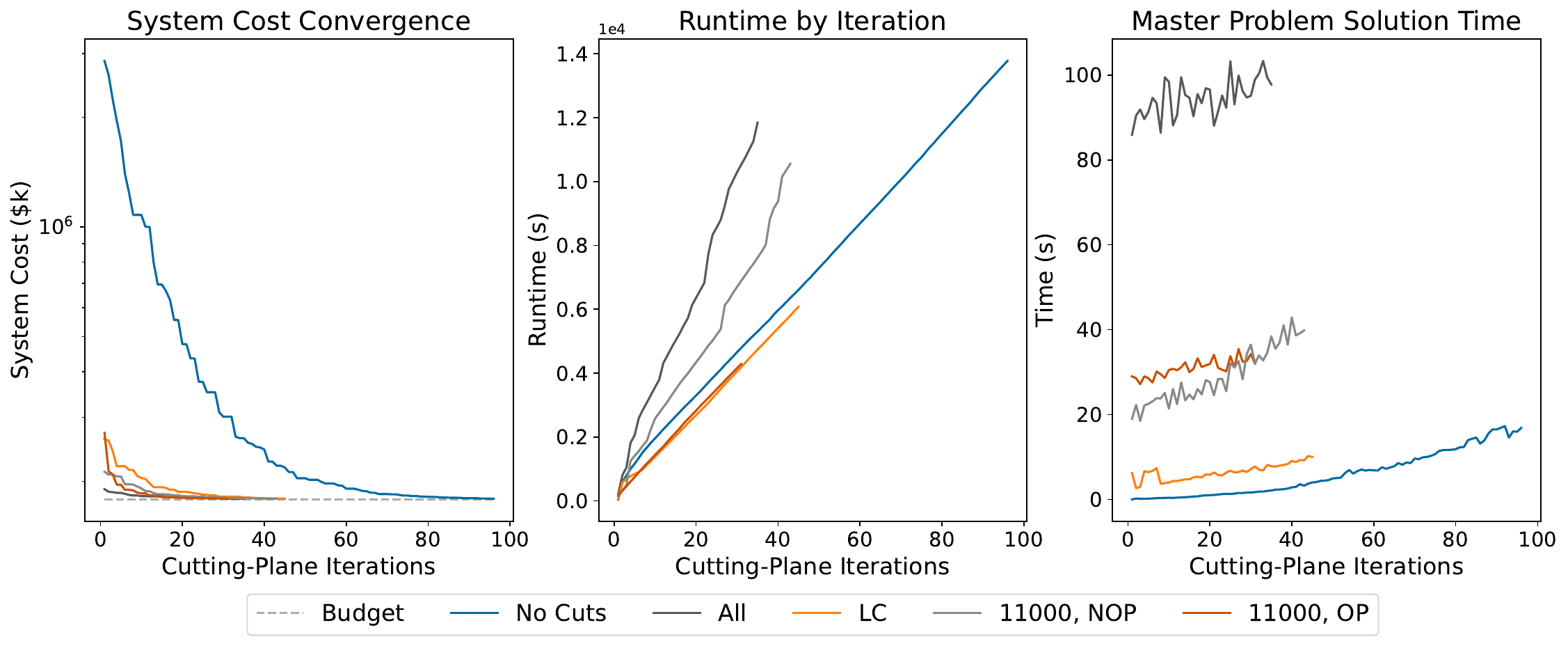}
    \caption{Comparison of algorithm convergence, runtime, and master problem solve time characteristics for a single MGA solution for all included cut sharing and objective partitioning methodologies. From the left, plot 1 shows convergence as measured by system cost declines, plot 2 shows runtime as a function of cutting-plane iterations, and plot 3 shows the progression of master problem solve times with cutting-plane iterations. We note that all methods which share cuts from previous iterations have meaningful advantages in initial cost gap estimates and runtime, but take longer to solve  master problems.}
    \label{fig:CutConv}
\end{figure}
\paragraph{}{Figure \ref{fig:CutConv} highlights that sharing least-cost cuts dramatically improves initial master problem solutions, as measured by \text{UB} budget violation. On average, improved early master problem solutions allow for convergence in less than half of the iterations of the control case without cut sharing, though individual problems vary. Despite having significantly better initial information, sharing all cuts takes nearly as many iterations to converge (in this case) as the least-cost cut sharing method, indicating there is not much improvement to be gained by incorporating significantly more cuts. We hypothesize that this is because the value of retained cuts decreases as the number of retained cuts increases. The combination of limiting cut numbers and objective partitioning by clustering similar vectors used in the third method with 11000 cuts does seem to show some benefit, as it has more limited master problem solution times than the all-cuts method and converges more quickly than the least-cost cuts method.}
\subsubsection{Integer Investment Decisions}
\paragraph{}{The testing in the prior two subsections is carried out on linear programs. All current large-scale MGA modelling is carried out using exclusively linear problems since it is computationally intractable to monolithically solve large MGA problems with mixed integer formulations. Mixed-integer problems are significantly more computationally intensive than linear problems as a class, and that difficulty is only increased by the addition of tight budget constraints and the requirement of many problem solves in the MGA context. }
\paragraph{}{In this subsection, we provide indicative testing of the computational performance of CGA on the previously described 26-zone model with integer investment constraints enabled. The resulting model is identical to the model described in Table \ref{tab:Tab1}, but with 2467 integer decisions in the master problem. We run the 26-zone problem for a least-cost solution and 9 MGA iterations with CGA using the least-cost cuts cut sharing methodology to provide some indicative runtime data. We attempted to run the same model monolithically with Gurobi with 32 cores, but the least-cost problem was unable to converge within a 48-hour limit. The resulting data are reported in Table \ref{tab:Integer}.}
\begin{table}[H]
    \centering
    \caption{Results: 26-zone Integer Problem solved with CGA}
    \label{tab:Integer}
    \begin{tabular}{c|cc}
    
    \textit{Metric} & Mean Sol. Time (s) & Iterations \\
    \midrule
    \textit{Mean} &  13534 & 49.3  \\
    \textit{Max} & 23261 & 70 \\
    \textit{Min} & 4975& 39 \\
       
    \end{tabular}
\end{table}
\paragraph{}{For comparison, the CGA algorithm solved the integer problem presented here faster (13534 seconds) on average than Gurobi solved the monolithic linear relaxation of the problem (20112 seconds).}
\section{Conclusion}\label{conc}
\paragraph{}{For more than a decade, MGA has been mentioned as an important method for exploring uncertainty in stakeholder objectives and problem structure and incorporating the messiness of the real-world in optimization problems through systematic exploration of near-optimal feasible solutions \cite{DeCarolis2017, DeCarolis2011, Trutnevyte2016}. By exploring the near-optimal feasible space, MGA could potentially provide decision-makers and stakeholders with a fuller range of within-budget portfolios, allowing a better understanding of the options available to them and therefore improving decision-making. Despite these strengths, MGA has not been widely adopted. The computational intensity of the method has to date been a major obstacle to the adoption of MGA. As we demonstrated in subsection \ref{mono}, MGA iterates are individually harder to solve than least-cost problems of the same dimensionality, and hundreds of MGA iterates must typically be solved to generate a suitably extensive exploration of the range of alternative near-optimal solution \cite{Lau2024}. Hence, they become computationally intractable unless spatial and temporal resolution is significantly reduced and integrality constraints relaxed relative to what one would run in a least-cost problem for the same case study. Computational acceleration of the MGA algorithm for large-scale energy system planning models is therefore an important step as computational intensity has limited MGA adoption in real-world applications.}
\paragraph{}{In this paper, we introduce a Parallelizable Cutting-Plane Algorithm to Generate Alternatives, a tailored, parallelizable algorithm with three cut-sharing techniques and a objective partitioning methodology for MGA problems in capacity expansion models. We prove the equivalence of the cutting-plane formulation, explain the algorithm, and introduce initialization techniques that share cuts from prior solutions to accelerate convergence in Section \ref{Sec2}. We then compare the cutting-plane algorithm to monolithic MGA in a series of numerical tests in Section \ref{sec:numexp}. We demonstrated that CGA is faster and more memory efficient than existing monolithic MGA implementations, finding that our cutting-plane algorithm replicates the results of monolithic MGA, but solves linear MGA iterates at least 2.8 times faster on an individual solution basis than monolithic MGA in high temporal and zonal resolution capacity expansion model applications. We demonstrate that initializing the master problem with cuts from previous MGA iterations and partitioning the MGA objective vector space are advantageous for our cutting-plane algorithm. Finally, we show that our cutting-plane algorithm makes solving large, spatially and temporally-detailed mixed integer MGA problems possible. The results presented herein suggest that the Cutting-Plane to Generate Alternatives (CGA) algorithm can enable the use of MGA for case studies with much higher spatial and temporal resolution and more realistic investment decision representation than was previously possible, limiting aggregation and approximation errors compared to existing solution techniques.}
\paragraph{}{This work makes the following four contributions:}
\begin{enumerate}
\item \textit{Parallelizable Cutting-Plane Algorithm to Generate Alternatives}: We developed a modified Benders decomposition algorithm convergence criteria to adapt Benders decomposition to MGA style problems, where cuts are used to approximate compliance with the MGA budget slack constraint, rather than the objective function. We find that this algorithm is faster and more memory-efficient than existing monolithic MGA algorithms on both linear and mixed-integer problems.
    \item \textit{Proof of equivalence of Cutting-Plane and monolithic MGA formulations}: We prove that monolithic MGA formulations, including the budget constraint, can be equivalently reformulated into a master and subproblem for use in a decomposed cutting-plane algorithm. We demonstrate that the cutting-plane algorithm implemented returns identical solutions to an equivalent monolithic formulation.
    \item \textit{Cut Sharing and Objective Partitioning}: We develop the concept of sharing cuts between problems which only vary in master problem objective. Furthermore, we introduce a method to partition sets of MGA search vectors to maximize the computational advantage of shared cuts. We then demonstrate that sharing cuts between problems and partitioning objective vectors provides computational benefits.
\end{enumerate}
\paragraph{}{There are several avenues for future work based on this contribution. First, we did not explore methods to determine which cuts to share beyond simple heuristics. Improved cut valuation methods would be of great benefit to this and other Benders decomposition-based algorithms. Second, other methods of clustering or partitioning MGA vectors could be valuable. While this work is focused on the development of the Cutting-Plane to Generate Alternatives algorithm in electricity system capacity expansion models, this algorithm is also broadly extensible to planning models with strict operational detail (e.g. multi-sectoral energy models) which solve problems repeatedly with changes limited to investment-related variables, coefficients, and constraints as they share much of the same problem structure. }

\printbibliography

\end{document}